\documentclass{amsart}

\usepackage{amssymb,amsmath,amsthm,latexsym,amscd}
\usepackage{epsfig}
\usepackage{psfrag}
\usepackage{graphicx}
\usepackage{bbold}

\newtheorem{Theorem}{\bf Theorem}
\newtheorem{lemma}[Theorem]{\bf Lemma}
\newtheorem{proposition}[Theorem]{\bf Proposition}

\newtheorem{theorem}[Theorem]{\bf Theorem}

\def\qed{\hfill$\Box$}
\def\scfig #1 #2 {\resizebox{#2}{!}{\includegraphics{#1}}}
\def\k{{\bf k}}

\newcommand{\be}{\begin{equation}}
\newcommand{\ee}{\end{equation}}

\def\hpic #1 #2 {\mbox{$\begin{array}[c]{l} 
\epsfig{file=#1,height=#2}\end{array}$}}
\def\wpic #1 #2 {\mbox{$\begin{array}[c]{l} 
\epsfig{file=#1,width=#2}\end{array}$}}

\begin{document}

\title[Planar algebras, cabling and the Drinfeld double]{Planar algebras, cabling and the Drinfeld double}

\author{Sandipan De}
\author{Vijay Kodiyalam}
\address{The Institute of Mathematical Sciences, Chennai, India}
\email{sandipande@imsc.res.in,vijay@imsc.res.in}


\subjclass[2010]{Primary 16T05, 16S40, 46L37}


\begin{abstract} We produce an explicit embedding of the planar algebra of the Drinfeld double of
a finite-dimensional, semisimple and cosemisimple Hopf algebra $H$ into
the two-cabling of the planar algebra of the dual Hopf algebra $H^*$
 and characterise the image.
\end{abstract}
\maketitle

\section{Introduction}

This paper is motivated by the main result of \cite{DeKdy2015} which shows
that a certain inclusion of infinite iterated crossed product algebras associated to a
finite-dimensional Hopf algebra $H$ is a crossed product by the Drinfeld double $D(H)$ of $H$. The hardest step in the proof was the identification of an explicit algebra embedding of $D(H)$ into the triple crossed product $H^* \rtimes H \rtimes H^*$. This required the machinery of Jones' planar algebras to
derive (while verifying it is merely an exercise in Sweedler's notation calisthenics).

For the rest of the paper, $\k$ will be an arbitrary algebraically closed field and $H = (H,\mu,\eta,\Delta,\epsilon,S)$ a finite-dimensional, semisimple and cosemisimple Hopf algebra over $\k$. Associated to $H$ is a planar algebra over $\k$ denoted by $P = P(H)$. The algebras $P_{0,\pm}$ and $P_{1,\pm}$ are identified with $\k$ and for $k \geq 2$, $P_{k,+} \cong H \rtimes H^* \rtimes H \rtimes  \cdots$
and $P_{k,-} \cong  H^* \rtimes H \rtimes  H^* \rtimes \cdots$ as algebras (where there are $k-1$ alternating factors, and for the natural actions of $H^*$ on $H$ and $H$ on $H^*$  defined by $f.a = f(a_2)a_1$ and $a.f = f_2(a)f_1$).

The embedding of $D(H)$ in $H^* \rtimes H \rtimes H^*$ may therefore
be regarded as a map of $P(D(H))_{2,+}$ into $P(H^*)_{4,+}$. Such maps are interesting for various reasons. For instance, one such embedding of $D(H)$
into the tensor square of $H \rtimes H^*$ is discussed in \cite{Ksh1996} and used in \cite{Ksh2011} to construct knot invariants in intrinsically three-dimensional terms. We note that $P(H^*)_{4,+}$ can be naturally regarded as a subalgebra of $P(H^*)_{5,+}$ which can be
identified with the tensor square of $H \rtimes H^*$ - both these being matrix algebras over $\k$ of size $dim(H)^2$.

It is thus a natural question to ask whether the embedding of $D(H)$
into $H^* \rtimes H \rtimes H^*$ may be extended to a planar algebra map
in some canonical fashion, and it is the affirmative answer to this question that is one of the main results of our paper. The other parts of the paper show that this planar algebra map is an embedding and characterise
the image.

\section{Preliminaries}

We review some required preliminaries before stating and proving the main theorem. 

\subsection{Semisimple and cosemisimple Hopf algebras}
The following facts about semisimple and cosemisimple Hopf algebras (over an algebraically closed field $\k$) may be found in \cite{TngGlk1998} and in \cite{LrsRdf1988}.
They are multi-matrix algebras whose dimensions, along with those of their
irreducible representations, are invertible in $\k$. The normalised (to be $1$ at the unit) traces in the
left regular representations of $H$ and $H^*$ will be denoted by $p$ and $h$ respectively. These are non-degenerate traces and are two-sided integrals, i.e.,
$hx = \epsilon(x)h = xh$ and $qp = q(1)p =pq$ for all $x \in H, q \in H^*$.
Also 
$p(h) = \frac{1}{n}$ where $n = dim(H)$. The antipodes of $H$ and $H^*$ are involutive. Some formulae involving integrals that we will use without
explicit mention are $h_1 \otimes h_2 = h_2 \otimes h_1 = Sh_1 \otimes Sh_2 = Sh_2 \otimes Sh_1$ and $xh_1 \otimes h_2 = h_1 \otimes Sxh_2$ along with the analogous formulae for $H^*$.

\subsection{Planar algebras} The notion of planar algebras has been evolving since the first paper \cite{Jns1999}. We elaborate a little
on the notion of planar algebras that we use in this paper which is different in some respects from the planar algebras that we have used in our previous work.
Assuming familiarity with some notion of planar algebras, we will be brief - see \cite{KdySnd2004} for a more leisurely discussion of the older version of planar algebras and \cite{Ghs2011} for the newer version.

Recall that the coloured operad of planar tangles underlies planar algebras.
Consider the set $Col = \{0,1,2,\cdots\} \times \{\pm 1 \}$, elements of which we refer to as colours. We will typically write a colour as $(k,\epsilon)$ where $\epsilon$ is either $+$ or $-$ and stands for $+1$ or $-1$.
A tangle is a subset of the plane that is the complement of the union of the interiors of a (possibly empty) finite numbered collection of 
internal discs in an external disc, along with the following data.
%
%
%
Each disc has an even number (again, possibly 
0) of points marked on its boundary circles.
There is also given a collection of disjoint curves on the tangle each of which is either a simple
closed curve,
or joins a marked point on one of the circles to another such, cutting each transversally.
Each marked point on a disc must be
the end-point of one of the curves.
For each disc,
 one
of its boundary arcs (= connected components of the complement of the marked points on the boundary circle) is distinguished and marked
with a $*$ placed near it.
Finally, there is given a chequerboard shading of the regions (= connected components of the complement of the curves) such 
that across any curve, the shading toggles.
A disc with $2n$ points on its boundary is said to be an $(n,+)$ disc or an $(n,-)$ disc according as its $*$-arc abuts a white or black region.
A tangle is said to be an $(n,\epsilon)$-tangle if its external disc is of colour $(n,\epsilon)$.
Tangles are defined only up to
a planar isotopy preserving the $*$-arcs, the shading and the numbering of the 
internal discs.
As is usual, we will often refer to and draw the discs as boxes with their *-arcs being their leftmost arc and sometimes omit drawing the external disc/box.

Two basic operations that can be performed on tangles are that of renumbering the internal discs or of  substitution of one tangle into a disc of another, and the collection of tangles  along with these  operations is called the coloured operad of planar tangles.
A planar algebra $P$ is an algebra over this operad. Thus $P$ is a collection (i) of vector spaces indexed by $Col$ and (ii) linear maps $Z_T^P$
indexed by tangles, with the maps being compatible with renumbering and substitution and satisfying the so-called non-degeneracy axiom.

We will refer to planar algebras in the older sense as restricted planar algebras. For these, the set of colours is the subset $\{(0,\pm),(1,+),(2,+),\cdots\}$, all discs (with the
exception of $(0,-)$-discs) have $*$-arcs abutting white regions and $P$ is a collection of vector spaces
indexed only by the subset above.
Clearly, a planar algebra naturally yields a restricted planar algebra (which we will refer to as its restriction) in the obvious manner. The converse holds too in the following categorical form - see Remark 3.6 of \cite{Ghs2011} (which treats the case when $P$ has modulus).

\begin{proposition}\label{restricted} Let $Q$ be a restricted planar algebra. There exists a
planar algebra $P$ with restriction isomorphic to $Q$. Further $P$ is unique
in the sense that if $P^1$ and $P^2$ are planar algebras with restrictions
$Q^1$ and $Q^2$ that are isomorphic (as restricted planar algebras) by the map $\phi: Q^1 \rightarrow Q^2$, then, there exists a unique planar algebra isomorphism $\tilde{\phi}: P^1 \rightarrow P^2$ that restricts to
$\phi$.
\end{proposition}

\begin{proof}[Sketch of proof] For existence, given $Q$, construct $P$ as follows. Define $P_{0,\pm} = Q_{0,\pm}$ and for $k > 0$, set $P_{k,\pm} = Q_{k,+}$. 

To define the action by tangles, first consider for every colour $(k,\epsilon)$,
the tangle $C^{(k,\epsilon)}$ which is defined to be the identity tangle of colour $(k,\epsilon)$ if $(k,\epsilon) \in \{(0,\pm),(1,+),(2,+),\cdots\}$ and to be the one-rotation tangle with an internal disc of colour $(k,+)$ and
external disc of colour $(k,-)$ otherwise. Also consider the tangle $D_{(k,\epsilon)}$  which is defined to be the identity tangle of colour $(k,\epsilon)$ if $(k,\epsilon) \in \{(0,\pm),(1,+),(2,+),\cdots\}$ and to be the inverse
one-rotation tangle with an internal disc of colour $(k,-)$ and
external disc of colour $(k,+)$ otherwise.

Now, for a $(k_0,\epsilon_0)$-tangle $T$ with internal discs of colours $(k_1,\epsilon_1),
\cdots,(k_b,\epsilon_b)$, define $Z_T^P = Z_{\tilde{T}}^Q$ where $\tilde{T} = D_{(k_0,\epsilon_0)} \circ T \circ_{(D_1,\cdots,D_b)} (C^{(k_1,\epsilon_1)},C^{(k_2,\epsilon_2)},\cdots,C^{(k_b,\epsilon_b)})$.
It is then easy to see that this defines a planar algebra structure on $P$ (with restriction $Q$),
the main observation being that $C^{(k,\epsilon)} \circ D_{(k,\epsilon)}$ is the identity tangle of colour $(k,\epsilon)$.

For uniqueness of $P$, suppose that $P^1,P^2$ are planar algebras with restrictions $Q^1,Q^2$ and $\phi: Q^1 \rightarrow Q^2$ is
a restricted planar algebra isomorphism. We need to see the existence 
and uniqueness  of a unique planar algebra isomorphism $\tilde{\phi}:   P^1 \rightarrow P^2$ that restricts to $\phi$.

The uniqueness of $\tilde{\phi}$ is because, given a colour $(k,-)$ with $k>0$, the
equation $\tilde{\phi}_{k,-} \circ Z^{P^1}_{C^{(k,-)}} = Z^{P^2}_{C^{(k,-)}}
\circ \phi_{k,+}$, must hold and $Z^{P}_{C^{(k,-)}}$
is an isomorphism for any planar algebra $P$. As for existence, define $\tilde{\phi}_{k,-}$ by the same
equation and check that this indeed gives a planar algebra isomorphism that restricts to $\phi$.
\end{proof}

%
%
%

\subsection{The planar algebra of a Hopf algebra}
Next, we recall the construction from \cite{KdySnd2006} of the planar algebra $P(H)$ associated to $H$. This depends on the choice of a square root, denoted $\delta$, of $n$ in $k$, which we will assume has been made and is fixed throughout. The planar algebra $P(H)$ is then defined to be
the quotient of the universal planar algebra on the label set $L = L_{2,+} = H$ by the set of relations in Figures \ref{fig:LnrMdl} - 
\ref{fig:XchNtp}.

\begin{figure}[!h]
\begin{center}
\psfrag{zab}{\huge $\zeta a + b$}
\psfrag{eq}{\huge $=$}
\psfrag{a}{\huge $a$}
\psfrag{b}{\huge $b$}
\psfrag{z}{\huge $\zeta$}
\psfrag{+}{\huge $+$}
\psfrag{del}{\huge $\delta$}
\resizebox{10.0cm}{!}{\includegraphics{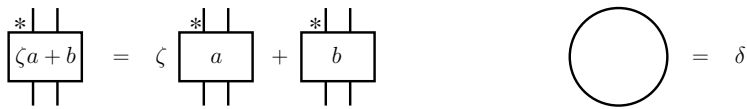}}
\end{center}
\caption{The L(inearity) and M(odulus) relations}
\label{fig:LnrMdl}
\end{figure}

\begin{figure}[!h]
\begin{center}
\psfrag{zab}{\huge $\zeta a + b$}
\psfrag{eq}{\huge $=$}
\psfrag{1h}{\huge $1_H$}
\psfrag{h}{\huge $h$}
\psfrag{z}{\huge $\zeta$}
\psfrag{+}{\huge $+$}
\psfrag{del}{\huge $\delta^{-1}$}
\resizebox{10.0cm}{!}{\includegraphics{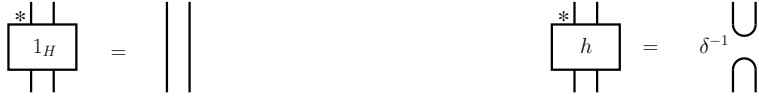}}
\end{center}
\caption{The U(nit) and I(ntegral) relations}
\label{fig:NitNtg}
\end{figure}

\begin{figure}[!h]
\begin{center}
\psfrag{epa}{\huge $\epsilon(a)$}
\psfrag{eq}{\huge $=$}
\psfrag{delinphia}{\huge $\delta p(a)$}
\psfrag{h}{\huge $h$}
\psfrag{a}{\huge $a$}
\psfrag{+}{\huge $+$}
\psfrag{del}{\huge $\delta^{-1}$}
\resizebox{10.0cm}{!}{\includegraphics{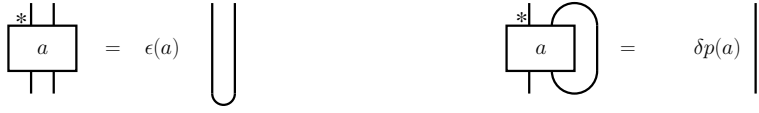}}
\end{center}
\caption{The C(ounit) and T(race) relations}
\label{fig:CntTrc}
\end{figure}

\begin{figure}[!h]
\begin{center}
\psfrag{epa}{\huge $\epsilon(a)$}
\psfrag{eq}{\huge $=$}
\psfrag{a1}{\huge $a_1$}
\psfrag{a2b}{\huge $a_2\,b$}
\psfrag{b}{\huge $b$}
\psfrag{a}{\huge $a$}
\psfrag{sa}{\huge $Sa$}
\psfrag{del}{\huge $\delta$}
\resizebox{10.0cm}{!}{\includegraphics{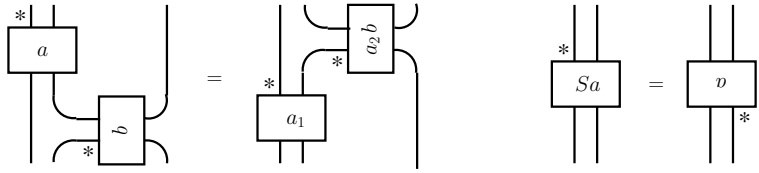}}
\end{center}
\caption{The E(xchange) and A(ntipode) relations}
\label{fig:XchNtp}
\end{figure}

In these figures, note that the shading is such that all the 2-boxes that occur are of colour $(2,+)$. Also note that the modulus relation is a pair
of relations - one for each choice of shading the circle.
%
%
%
%
Finally, note that the interchange of $\delta$ and $\delta^{-1}$ between the (I) and (T) relations here and those of \cite{KdySnd2006} is due to the different
normalisations of $h$ and $p$.


The main result of \cite{KdySnd2006} then asserts that $P(H)$ is a connected, irreducible, spherical, non-degenerate planar algebra with modulus $\delta$ and is of depth two with $dim(P(H)_{k,\pm}) = n^{k-1}$ for every $k \geq 1$. The word `non-degenerate' here refers not to the non-degeneracy axiom (which must hold for any planar algebra) but to the condition that the
trace tangles for each colour specify non-degenerate traces.

\subsection{Cabling} For any positive integer $m$, consider the `operation $T \mapsto T^{(m)}$ on
tangles' given by $m$-cabling. Some care is needed in defining this for tangles involving $(k,\epsilon)$ boxes with $\epsilon = -1$. Take the tangle
$T$, ignore its shading and thicken every strand to a cable of $m$ parallel
strands without changing the $*$'s. Now introduce shading in the result such that any $(k,\epsilon)$ box of $T$ changes to a $(mk,\epsilon^m)$ box of $T^{(m)}$. A little thought shows that this does give a consistent chequerboard shading as needed.
For a detailed definition of `operation on tangles', see \cite{KdySnd2004}.
This gives an operation on planar algebras $P \mapsto  {}^{(m)}\!P$. Here, the planar algebra ${}^{(m)}\!P$ = $Q$, say, is defined by setting the vector spaces $Q_{k,\pm}$ to be $P_{mk,(\pm)^m}$ and the action $Z_T^{Q}$ of a tangle $T$ on $Q$ to be $Z_{T^{(m)}}^P$.

In this paper we will only be interested in the $2$-cabling of a planar algebra $P$ whose spaces are specified by $(^{(2)}\!P)_{k,\pm} = P_{2k,+}$
for any $k \in \{0,1,2,\cdots\}$.

\subsection{The Drinfeld double} For a finite-dimensional Hopf algebra $H$, its Drinfeld double, denoted $D(H)$, is the Hopf algebra with underlying vector space $H^* \otimes H$ and multiplication, comultiplication and antipode specified by the following formulae:
\begin{eqnarray*}
(f \otimes x)(g \otimes y) &=& g_1(Sx_1)g_3(x_3) (g_2f \otimes x_2y),\\
\Delta(f \otimes x) &=& (f_1 \otimes x_1) \otimes (f_2 \otimes x_2), {\mbox { and}}\\
S(f \otimes x) &=& f_1(x_1)f_3(Sx_3) (S^{-1}f_2 \otimes Sx_2).
\end{eqnarray*}

As in \cite{DeKdy2015} what we will actually use is an isomorphic version of $D(H)$ which we denote $\tilde{D}(H)$ which also has underlying vector space $H^* \otimes H$ and structure maps obtained by transporting the structures on $D(H)$ using the invertible map $S \otimes S^{-1} : D(H) = H^* \otimes H \rightarrow H^* \otimes H = \tilde{D}(H)$. It is easily checked that the structure maps for $\tilde{D}(H)$ are given by the following formulae.
\begin{eqnarray*}
(f \otimes x)(g \otimes y) &=& g_1(x_1)g_3(Sx_3) (fg_2 \otimes yx_2),\\
\Delta(f \otimes x) &=& (f_2 \otimes x_2) \otimes (f_1 \otimes x_1), {\mbox { and}}\\
S(f \otimes x) &=& f_1(Sx_1)f_3(x_3) (S^{-1}f_2 \otimes Sx_2).
\end{eqnarray*}

The unit and counit of $\tilde{D}(H)$ are given by $\epsilon \otimes 1$ and
$1 \otimes \epsilon$ respectively.
If $H$ is semisimple and cosemisimple, 
there are normalised two-sided integrals
$p \otimes h \in \tilde{D}(H)$ and $h \otimes p \in \tilde{D}(H)^*$, and
in particular,  $D(H)$ is both semisimple and cosemisimple. Also, in this case, since the antipode $S$ is involutive, there is no distinction between $S$ and $S^{-1}$.

\section{The planar algebra morphism}

The following proposition is the first part of the main result of this paper.

\begin{proposition}\label{main}  Let $H$ be a finite-dimensional, semisimple and cosemisimple Hopf algebra over $\k$ of dimension $n = \delta^2$
with Drinfeld double $\tilde{D}(H)$. The map
$$\tilde{D}(H) \cong P(\tilde{D}(H))_{2,+}\  \longrightarrow\  {}^{(2)}\!P(H^*)_{2,+} = P(H^*)_{4,+} \cong H^* \rtimes H \rtimes H^*$$
defined by linear extension of $f \otimes a \mapsto f_1(Sa_1) f_3 \rtimes Sa_2 \rtimes f_2$
extends to a unique planar algebra morphism from $P(\tilde{D}(H))$
to $^{(2)}\!P(H^*)$.
\end{proposition}

Before beginning the proof, we clarify, in the following lemma whose proof we omit, the isomorphisms occurring in the
statement of the proposition.
When $H$ is a Kac algebra and for the $P_{k,+}$, the proof, in detail, appears in the thesis \cite{Jjo2008}. The proof in our case is identical and proceeds by induction on $k$. The statement uses the Fourier transform map $F_H: H \rightarrow H^*$ defined by $F_H(a) = \delta p_1(a) p_2$ which satisfies $F_{H^*}F_H = S$. We will usually omit the subscript of $F_H$ and $F_{H^*}$ and write both as $F$ with the argument making it clear which is meant.

\begin{lemma}\label{lemma:isom} Let $H$ be a finite-dimensional, semisimple and cosemisimple Hopf algebra over $\k$ with planar algebra $P(H)$.
For each $k \geq 2$ the maps
\begin{eqnarray*}
H \rtimes H^* \rtimes H \rtimes \cdots (k-1 {\rm ~factors})&\longrightarrow& P(H)_{k,+} {\rm ~and}\\
H^* \rtimes H \rtimes H^* \rtimes \cdots (k-1 {\rm ~factors})&\longrightarrow& P(H)_{k,-},
\end{eqnarray*}
defined as in Figure \ref{fig:isom} are algebra isomorphisms.\qed
\begin{figure}[!h]
\psfrag{a}{\huge $a$}
\psfrag{Ff}{\huge $Ff$}
\psfrag{Fg}{\huge $Fg$}
\psfrag{b}{\huge $b$}
\psfrag{afb}{\huge $a \rtimes f \rtimes b \rtimes \cdots ~~ \mapsto$}
\psfrag{fag}{\huge $f \rtimes a \rtimes g \rtimes \cdots ~~ \mapsto$}
\psfrag{ddots}{\huge $\cdots$}
\begin{center}
\resizebox{13cm}{!}{\includegraphics{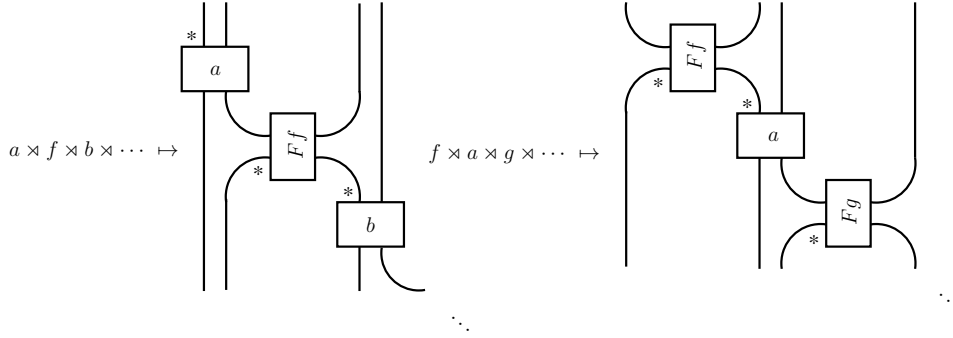}}
\end{center}
\caption{Algebra isomorphisms}
\label{fig:isom}
\end{figure}
\end{lemma}

The idea of the proof of Proposition \ref{main} is very simple. Since we know a presentation of the planar algebra of $\tilde{D}(H)$ by generators and relations, in order to define a planar algebra map of
$P(\tilde{D}(H))$ into any planar algebra, it suffices to map the generators to suitable elements in the target planar algebra in such a way that the relations hold. 

\begin{proof}[Proof of Proposition \ref{main}] Throughout this proof, we will use
$P$ to denote the planar algebra $P(\tilde{D}(H))$.

The map defined in the statement of Proposition \ref{main} can also be
expressed as $f \otimes a \mapsto (f_2 \rtimes 1 \rtimes f_1)(\epsilon \rtimes S(a) \rtimes \epsilon)$, as a brief calculation shows - see the  pictorial  rule for multiplication in iterated cross products in \cite{DeKdy2015}. This map is
shown pictorially in Figure \ref{fig:mapping0}.
Being bilinear in $f$ and $a$, this map
\begin{figure}[!h]
\psfrag{c}{\huge $ $}
\psfrag{a}{\huge $f_2$}
\psfrag{Ff}{\huge $FSa$}
\psfrag{Fg}{\huge $F$}
\psfrag{b}{\huge $f_1$}
\psfrag{afb}{\huge $f \otimes a \ \ \ \ \ \  \mapsto$}
\psfrag{fag}{\huge $f \rtimes a \rtimes g \rtimes \cdots ~~ \mapsto$}
\psfrag{ddots}{\huge $\cdots$}
\begin{center}
\resizebox{6cm}{!}{\includegraphics{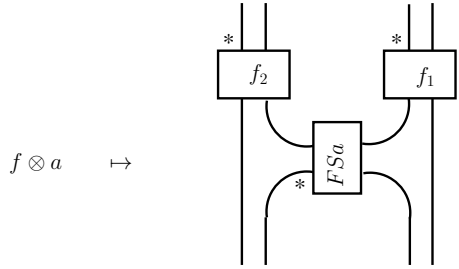}}
\end{center}
\caption{Mapping $\tilde{D}(H)$ to $P_{4,+}(H^*)$}
\label{fig:mapping0}
\end{figure}
clearly admits a linear extension to a map $\tilde{D}(H) \longrightarrow {}P(H^*)_{4,+} = {}^{(2)}\!P(H^*)_{2,+}$. Consider its extension to a planar algebra map from the universal planar algebra on $L = L_{2,+} = \tilde{D}(H)$ to ${}^{(2)}\!P(H^*)$. We will now check that each of the 8 relations (L), (M), (U), (I), (C), (T), (E), (A) in Figures \ref{fig:LnrMdl} - \ref{fig:XchNtp} (applied to the Hopf algebra $\tilde{D}(H)$) is in the kernel of this planar algebra map.\\
{\em\bf Relation L}: This is a direct consequence of the linearity of the map
$\tilde{D}(H) \longrightarrow {}^{(2)}\!P(H^*)_{2,+} = P(H^*)_{4,+}$ together with the multilinearity of tangle maps.\\
{\em\bf Relations M}: The modulus relations for $P= P(\tilde{D}(H))$ depend on a choice of square root of $dim(\tilde{D}(H)) = n^2$ and we will choose $n$ to be the modulus of $P$. Thus the modulus relations for
$P$ assert that $Z_{T^{0,\pm}}^P(1) = n 1_{0,\pm}$ where $T^{0,\pm}$ are the $(0,\pm)$ tangles with just one internal closed loop
and no internal discs and $1_{0,\pm}$ are the unit elements of $P_{0,\pm}$.
Pushing this down to ${}^{(2)}\!P(H^*)$, what needs to be verified is that, $Z_{T^{0,\pm}}^{{}^{(2)}\!P(H^*)}(1) = n 1_{0,\pm}$ or equivalently that
$Z_{(T^{0,\pm})^{(2)}}^{P(H^*)}(1) = n 1_{0,+}$.

Since the 2-cabled tangle $(T^{0,\pm})^{(2)}$ is just the $(0,+)$ tangle with
two parallel internal closed loops (and no internal discs), the asserted equality is a consequence of (one application of each of) the two modulus relations for $P(H^*)$.
\\
{\em\bf Relation U}: This is the equality $Z_{I^{2,+}_{2,+}}^P(1_{\tilde{D}(H)}) = Z_{U^{2,+}}^P(1)$, where
$I^{2,+}_{2,+}$ and $U^{2,+}$ are the identity and unit tangles of colour
$(2,+)$. In order to push this down to ${}^{(2)}\!P(H^*)$, we note first that
$1_{\tilde{D}(H)} = \epsilon \otimes 1$ ($= \epsilon \otimes 1_H$) and that under the map of Figure \ref{fig:mapping0} it goes to $1_{4,+}$ - the unit element of $P_{4,+}(H^*)$. This is because $FS(1) = F(1) = \delta p$ and by use of the integral relation in $P(H^*)$.

Thus what needs to be verified is that $Z_{I^{2,+}_{2,+}}^{{}^{(2)}\!P(H^*)}(1_{4,+}) = Z_{U^{2,+}}^{{}^{(2)}\!P(H^*)}(1)$ or equivalently that $Z_{(I^{2,+}_{2,+})^{(2)}}^{P(H^*)}(1_{4,+}) = Z_{(U^{2,+})^{(2)}}^{P(H^*)}(1)$. The last equality holds since $(I^{2,+}_{2,+})^{(2)} = I^{4,+}_{4,+}$,
$(U^{2,+})^{(2)} = U^{4,+}$ and, by definition, $1_{4,+} = Z_{U^{4,+}}(1)$.
\\
{\em\bf Relation I}: This is the equality $Z_{I^{2,+}_{2,+}}^P(h_{\tilde{D}(H)}) = n^{-1} Z_{E^{2,+}}^P(1)$, where
$h_{\tilde{D}(H)}$ is the integral in $\tilde{D}(H)$ and $E^{2,+}$ is the Jones projection tangle of colour
$(2,+)$. To push this down to ${}^{(2)}\!P(H^*)$, recall first that $h_{\tilde{D}(H)} =
p \otimes h$. Under the map of Figure \ref{fig:mapping0} this goes to the element of $P_{4,+}(H^*)$ shown on the left in Figure \ref{fig:mapping2},
\begin{figure}[!h]
\psfrag{c}{\huge $ $}
\psfrag{a}{\huge $p_2$}
\psfrag{Ff}{\huge $FSh$}
\psfrag{Fg}{\huge $F$}
\psfrag{b}{\huge $p_1$}
\psfrag{n-1}{\huge $n^{-1}$}
\psfrag{afb}{\huge $f \otimes a \ \ \ \ \ \  \mapsto$}
\psfrag{fag}{\huge $f \rtimes a \rtimes g \rtimes \cdots ~~ \mapsto$}
\psfrag{ddots}{\huge $\cdots$}
\begin{center}
\resizebox{8cm}{!}{\includegraphics{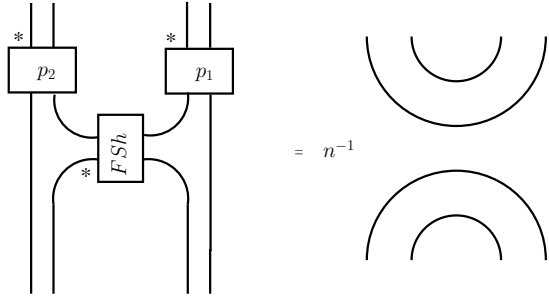}}
\end{center}
\caption{Equality to be verified in $P_{4,+}(H^*)$}
\label{fig:mapping2}
\end{figure}
which needs to be shown to be equal to $n^{-1} Z_{E^{2,+}}^{{}^{(2)}\!P(H^*)}(1) =
n^{-1} Z_{(E^{2,+})^{(2)}}^{P(H^*)}(1)$, which is the element of $P_{4,+}(H^*)$ shown on the right in Figure \ref{fig:mapping2}.

We prove this as follows. First note that $FS(h) = F(h) = \delta p_1(h)p_2 =
\delta p(h) \epsilon = \delta^{-1} \epsilon$. Now applying the unit relation in $P(H^*)$ we reduce the element on the left side of Figure \ref{fig:mapping2} to that on the left side in Figure \ref{fig:mapping3}.

Then we calculate in $P(H^*)$ as in Figure \ref{fig:mapping3}, where
\begin{figure}[!h]
\psfrag{c}{\huge $\delta^{-1}$}
\psfrag{a}{\huge $p_3$}
\psfrag{Ff}{\huge $p_1$}
\psfrag{p}{\huge $p$}
\psfrag{b}{\huge $p_2$}
\psfrag{n-1}{\huge $\delta^{-1}$}
\psfrag{n-2}{\huge $\delta^{-2}$}
\psfrag{afb}{\huge $f \otimes a \ \ \ \ \ \  \mapsto$}
\psfrag{fag}{\huge $f \rtimes a \rtimes g \rtimes \cdots ~~ \mapsto$}
\psfrag{ddots}{\huge $\cdots$}
\begin{center}
\resizebox{9cm}{!}{\includegraphics{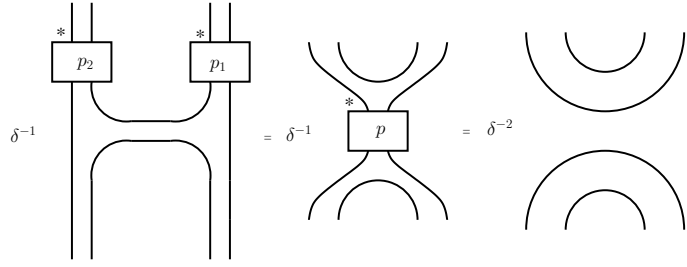}}
\end{center}
\caption{Computation in $P(H^*)$}
\label{fig:mapping3}
\end{figure}
the first equality follows from the exchange and antipode relations in $P(H^*)$ together with the fact that $Sp=p$ and the Hopf algebra identity $p_1 \otimes p_2 = p_2 \otimes p_1$ (which essentially expresses the traciality of $p$), while the second equality follows from the integral relation. Now, comparison with the previous step immediately yields the equality expressed in Figure \ref{fig:mapping2}, thus verifying Relation (I).\\
{\em\bf Relation C}: Recalling that the counit of $\tilde{D}(H)$ is given by $1 \otimes \epsilon$, the verification that Relation C is in the kernel of the planar algebra map from $P$ to ${}^{(2)}\!P(H^*)$ is easily seen to be equivalent to the truth of the equation of Figure \ref{fig:mapping4} holding in $P(H^*)$.

\begin{figure}[!h]
\psfrag{a}{\huge $f_2$}
\psfrag{Ff}{\Large $FS(a)$}
\psfrag{p}{\huge $p$}
\psfrag{b}{\huge $f_1$}
\psfrag{n-1}{\huge $f(1)\epsilon(a)$}
\psfrag{n-2}{\huge $\delta^{-2}$}
\psfrag{afb}{\huge $f \otimes a \ \ \ \ \ \  \mapsto$}
\psfrag{fag}{\huge $f \rtimes a \rtimes g \rtimes \cdots ~~ \mapsto$}
\psfrag{ddots}{\huge $\cdots$}
\begin{center}
\resizebox{7cm}{!}{\includegraphics{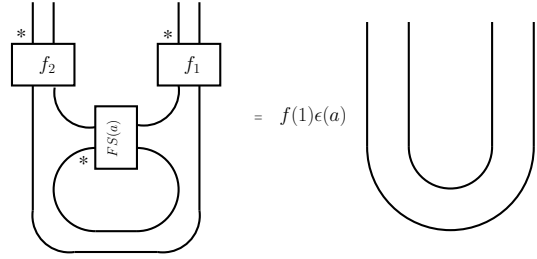}}
\end{center}
\caption{Relation C}
\label{fig:mapping4}
\end{figure}
This is easily seen since the trace and antipode relations in $P(H^*)$ simplify the looped $FS(a)$ to $\delta F(a)(h) =\delta^2 p(Sah) = \epsilon(a)$. Now $f_2Sf_1 = f(1) \epsilon$, so the required equality follows using Relation $U$.\\
{\em\bf Relation T}: Recalling that $p_{\tilde{D}(H)} = h \otimes p$, the verification that Relation T is in the kernel of the planar algebra map from $P$ to ${}^{(2)}\!P(H^*)$ is easily seen to be equivalent to the truth of the equation of Figure \ref{fig:mapping5} holding in $P(H^*)$.
\begin{figure}[!h]
\psfrag{a}{\huge $f_2$}
\psfrag{Ff}{\Large $FS(a)$}
\psfrag{p}{\huge $p$}
\psfrag{b}{\huge $f_1$}
\psfrag{n-1}{\huge $n f(h)p(a)$}
\psfrag{n-2}{\huge $\delta^{-2}$}
\psfrag{afb}{\huge $f \otimes a \ \ \ \ \ \  \mapsto$}
\psfrag{fag}{\huge $f \rtimes a \rtimes g \rtimes \cdots ~~ \mapsto$}
\psfrag{ddots}{\huge $\cdots$}
\begin{center}
\resizebox{6cm}{!}{\includegraphics{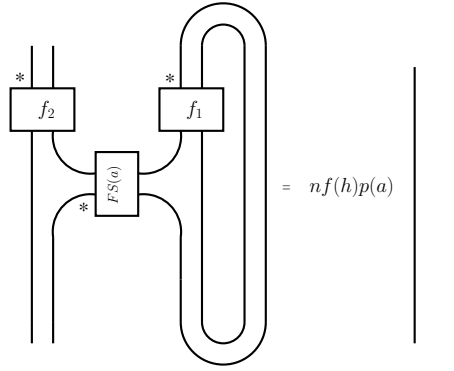}}
\end{center}
\caption{Relation T}
\label{fig:mapping5}
\end{figure}
The left hand side of Figure \ref{fig:mapping5} simplifies, using the trace and counit relations in $P(H^*)$, to $\delta f_1(h) FS(a)(1) f_2 = \delta^2 p(a) f_1(h)f_2 = np(a)f(h)\epsilon$, as needed.
\\
{\em\bf Relation E}: This is equivalent to two relations - one for multiplication and the other for comultiplication. These are shown in Figure \ref{fig:multcomult}.
\begin{figure}[!h]
\psfrag{a}{\huge $a$}
\psfrag{ab}{\huge $ab$}
\psfrag{p}{\huge $p$}
\psfrag{b}{\huge $b$}
\psfrag{a1}{\huge $a_1$}
\psfrag{a2}{\huge $a_2$}
\psfrag{afb}{\huge $f \otimes a \ \ \ \ \ \  \mapsto$}
\psfrag{fag}{\huge $f \rtimes a \rtimes g \rtimes \cdots ~~ \mapsto$}
\psfrag{ddots}{\huge $\cdots$}
\begin{center}
\resizebox{10cm}{!}{\includegraphics{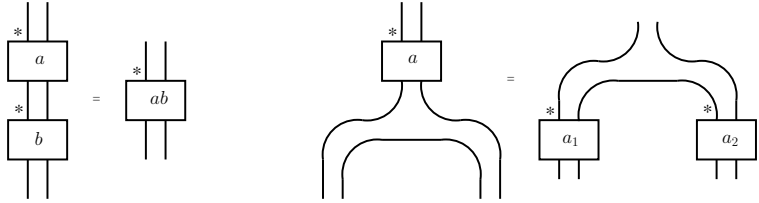}}
\end{center}
\caption{Multiplication and comultiplication relations}
\label{fig:multcomult}
\end{figure}
To prove that the multiplication relation is in the kernel, a little thought shows that it suffices to verify
the equality of Figure \ref{fig:multcheck} in $P(H^*)$.
\begin{figure}[!h]
\psfrag{f2}{\huge $f_2$}
\psfrag{f1}{\huge $f_1$}
\psfrag{fsa}{\huge $FSa$}
\psfrag{f3}{\huge $f_3$}
\psfrag{fsa2}{\huge $FSa_2$}
\psfrag{a2}{\huge $a_2$}
\psfrag{long}{\huge $=f_1(a_1)f_4(Sa_3)$}
\psfrag{fag}{\huge $f \rtimes a \rtimes g \rtimes \cdots ~~ \mapsto$}
\psfrag{ddots}{\huge $\cdots$}
\begin{center}
\resizebox{8cm}{!}{\includegraphics{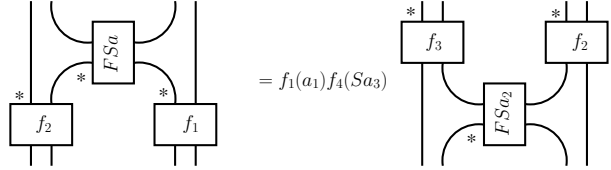}}
\end{center}
\caption{Equality to be verified for the multiplication relation}
\label{fig:multcheck}
\end{figure}
Using the exchange relation in $P(H^*)$ twice,  the equality in Figure \ref{fig:multcheck} is equivalent to the Hopf algebraic identity:
$$
f_1 \otimes f_2 \otimes FS(a)  = f_1(a_1)f_6(Sa_3) f_3 \otimes f_4 \otimes f_5SF(a_2)Sf_2.
$$
To see this, it certainly suffices to see that
$$
f \otimes FS(a)  = f_1(a_1)f_5(Sa_3) f_3 \otimes f_4SF(a_2)Sf_2.
$$
Evaluating both sides on an arbitrary element $x \otimes y \in H \otimes H$, we need to verify the equality
$$
f(x) FS(a)(y) = f_1(a_1)f_5(Sa_3) f_3(x) (f_4SF(a_2)Sf_2)(y).
$$
The right hand side of the above equation may be written as:
\begin{eqnarray*}
RHS &=& f_1(a_1)f_5(Sa_3) f_3(x)f_4(y_1)SF(a_2)(y_2)Sf_2(y_3)\\
&=& \delta f_1(a_1)f_5(Sa_3) f_3(x)f_4(y_1)p_1(Sa_2)p_2(y_2)Sf_2(y_3)\\
&=& \delta f_1(a_1)f_5(Sa_3) f_3(x)(f_4p_2Sf_2)(y)p_1(Sa_2)\\
&=& \delta f_1(a_1)f_5(Sa_3) f_3(x)p_2(y)(Sf_4p_1f_2)(Sa_2)\\
&=& \delta f_1(a_1)f_5(Sa_5) f_3(x)p_2(y)Sf_4(Sa_4)p_1(Sa_3)f_2(Sa_2)\\
&=& \delta f(a_1Sa_2xa_4Sa_5) p_2(y)p_1(Sa_3)\\
&=& \delta f(x) p(Say),
\end{eqnarray*}
which clearly agrees with the left hand side, finishing the proof of the multiplication relation.

Checking that the comultiplication relation is in the kernel is seen to be
equivalent to the identity of Figure \ref{fig:comultcheck} holding in $P(H^*)$.
\begin{figure}[!h]
\psfrag{f2}{\huge $f_2$}
\psfrag{f1}{\huge $f_1$}
\psfrag{fsa}{\huge $FSa$}
\psfrag{fsa1}{\huge $FSa_2$}
\psfrag{f3}{\huge $f_3$}
\psfrag{fsa2}{\huge $FSa_1$}
\psfrag{f4}{\huge $f_4$}
\psfrag{long}{\huge $=f_1(a_1)f_4(Sa_3)$}
\psfrag{fag}{\huge $f \rtimes a \rtimes g \rtimes \cdots ~~ \mapsto$}
\psfrag{ddots}{\huge $\cdots$}
\begin{center}
\resizebox{10cm}{!}{\includegraphics{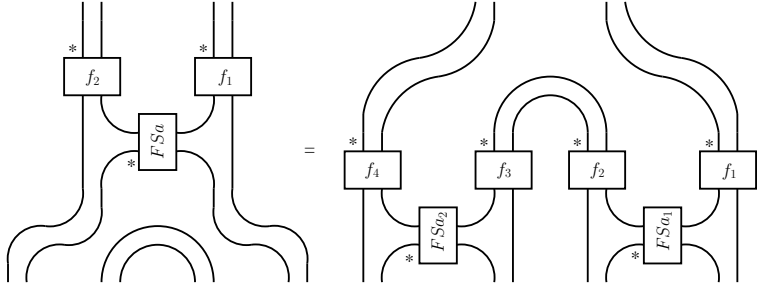}}
\end{center}
\caption{Equality to be verified for the comultiplication relation}
\label{fig:comultcheck}
\end{figure}
This easily reduces to verifying the relation of Figure \ref{fig:comultcheck2}.
We leave this pleasant verification to the reader.
\begin{figure}[!h]
\psfrag{f2}{\huge $f_2$}
\psfrag{f1}{\huge $f_1$}
\psfrag{fa}{\huge $Fa$}
\psfrag{Fa1}{\huge $Fa_2$}
\psfrag{f3}{\huge $f_3$}
\psfrag{Fa2}{\huge $Fa_1$}
\psfrag{f4}{\huge $f_4$}
\psfrag{long}{\huge $=f_1(a_1)f_4(Sa_3)$}
\psfrag{fag}{\huge $f \rtimes a \rtimes g \rtimes \cdots ~~ \mapsto$}
\psfrag{ddots}{\huge $\cdots$}
\psfrag{=}{\huge $=$}
\begin{center}
\resizebox{3.5cm}{!}{\includegraphics{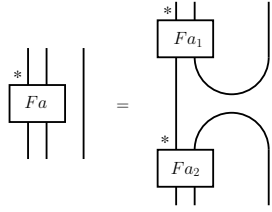}}
\end{center}
\caption{Equivalent equality to be verified}
\label{fig:comultcheck2}
\end{figure}
\\
{\em\bf Relation A}: The easiest way to see that the antipode relation is in the kernel of the planar algebra map from $P$ to ${}^{(2)}\!P(H^*)$ is to appeal
to the already proved multiplication relation. Since $S$ is an anti-homomorphism as is the 2-rotation, compatibility with multiplication immediately reduces to checking the antipode relation on an algebra generating set of $\tilde{D}(H)$. These may be chosen to be $f\otimes1$ and
$a \otimes \epsilon$ and on elements of this kind the antipode relation is
trivial to verify.
\end{proof}

\section{Injectivity}

In this section, we establish the injectivity of the planar algebra map obtained in the previous section.

\begin{proposition}\label{injective}
The planar algebra morphism $P(\tilde{D}(H))$ (=$P$) 
to $^{(2)}\!P(H^*)$ defined in Proposition \ref{main} is injective.
\end{proposition}

Before proceeding with the proof we recall a result from \cite{KdySnd2006}.
Let ${\mathcal T}(k,\epsilon)$ denote the set of $(k,\epsilon)$ tangles
with $k-1$ (interpreted as $0$ for $k=0$) internal boxes of colour $(2,+)$ and no `internal regions'. The result then asserts:

\begin{lemma}
[Lemma 16 of \cite{KdySnd2006}]
For each tangle $X \in {\mathcal T}(k,\epsilon)$, the map $Z_X^{P(H)}:
(P_{2,+}(H))^{\otimes(k-1)} \rightarrow P_{k,\epsilon}(H)$ is an isomorphism.
\end{lemma}

While the statement in \cite{KdySnd2006} assumes $\epsilon=1$ and $k \geq 3$, it is easy to see that neither restriction is really necessary.

\begin{proof}[Proof of Proposition \ref{injective}]
Let $\Psi: P \rightarrow \,^{(2)}\!P(H^*)$ denote the planar
algebra morphism of Proposition \ref{main}, which is a collection of maps
$\Psi_{k,\epsilon}: P_{k,\epsilon} \rightarrow (^{(2)}\!P(H^*))_{k,\epsilon} = P_{2k,+}(H^*)$ for each colour $(k,\epsilon)$. To see that each of these is injective, it suffices to check this when either $k=0$ or when $\epsilon=1$ (since the one-rotation tangles for $k >0$ give isomorphisms). The cases when $k=0$ (and $\epsilon = \pm1$) are obvious
since both sides are naturally isomorphic to $\bf k$ with the $\Psi_{0,\pm}$'s
reducing to the identity map under these isomorphisms. Also, $\Psi_{1,+}$
takes $1_1 \in P_{1,+}$ to $1_2 \in P_{2,+}(H^*)$, and is therefore injective.

For $k \geq 2$ (and $\epsilon=1$)
consider the family of tangles $X^{k,+}$ with $k-1$ internal boxes of colour $(2,+)$ defined inductively as in Figure \ref{fig:xkind}.
\begin{figure}[!h]
\begin{center}
\psfrag{xk}{\huge $X^{k,+}$}
\psfrag{kp1}{\huge $k$}
\psfrag{x2}{\huge $X^{2,+} =$}
\psfrag{xkp1}{\huge $X^{k+1,+} =$}
\resizebox{9cm}{!}{\includegraphics{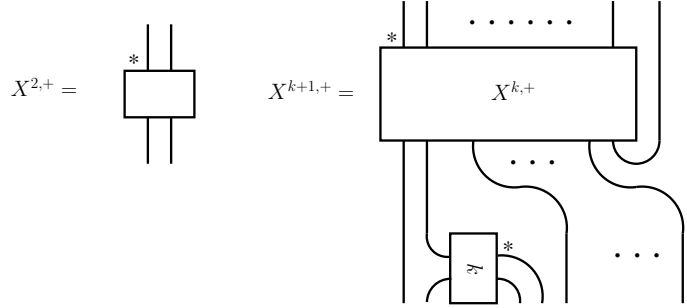}}
\end{center}
\caption{Inductive definition of $X^{k+1,+}$}
\label{fig:xkind}
\end{figure}
It is easy to see that $X^{k,+} \in {\mathcal T}(k,+)$. Thus $Z_{X^{k,+}}^{P} :  (\tilde{D}(H))^{\otimes (k-1)} \rightarrow P_{k,+}$ is an isomorphism,
and to show that $\Psi_{k,+}$ is injective it suffices to see that $\Psi_{k,+} \circ Z_{X^{k,+}}^{P}$ is injective.

Since $\Psi$ is a planar algebra morphism defined by the extension of the map of Figure \ref{fig:mapping0}, it follows easily that $\Psi_{k,\epsilon} \circ Z_{X^{k,+}}^{P}((f^1 \otimes x^1) \otimes (f^2 \otimes x^2) \otimes
\cdots \otimes (f^{(k-1)} \otimes x^{(k-1)}))$ is given by the element of $P_{2k,+}(H^*)$ shown in Figure \ref{fig:injec}.
\begin{figure}[!h]
\begin{center}
\psfrag{f11}{\huge $f^1_1$}
\psfrag{c}{\huge $\cdots$}
\psfrag{f12}{\huge $f^1_2$}
\psfrag{f21}{\huge $f^2_1$}
\psfrag{f22}{\huge $f^2_2$}
\psfrag{f1km1}{\huge $f^{k-1}_1$}
\psfrag{f2km1}{\huge $f^{k-1}_2$}
\psfrag{fsx1}{\Large $FSx^1$}
\psfrag{fsx2}{\Large $FSx^2$}
\psfrag{fsxkm1}{\Large $FSx^{k-1}$}
\resizebox{12cm}{!}{\includegraphics{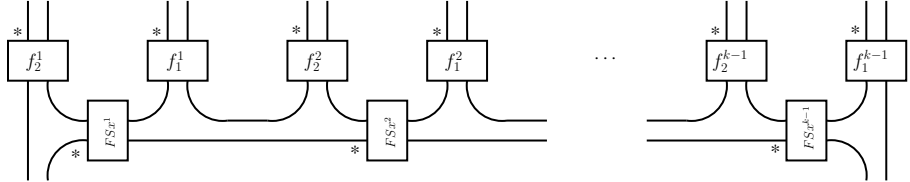}}
\end{center}
\caption{$\Psi_{k,\epsilon} \circ Z_{X^{k,+}}^{P}((f^1 \otimes x^1) \otimes (f^2 \otimes x^2) \otimes
\cdots \otimes (f^{(k-1)} \otimes x^{(k-1)}))$}
\label{fig:injec}
\end{figure}
Some manipulation with the exchange relation in $P(H^*)$ shows that this element is also equal to $Z_{S^{2k,+}}^{P(H^*)}
( f^1_2 \otimes FSx^1Sf^2_3 \otimes f^1_1Sf^2_2 \otimes f^2_4FSx^2Sf^3_3 \otimes f^2_1Sf^3_2 \otimes \cdots \otimes f^{k-2}_4FSx^{k-2}Sf^{k-1}_3 \otimes f^{k-2}_1Sf^{k-1}_2 \otimes f^{k-1}_4FSx^{k-1} \otimes f^{k-1}_1)$
 where $S^{2k,+}$ is the $(2k,+)$ tangle with $2k-1$ internal boxes of colour $(2,+)$ shown in Figure \ref{fig:Stangle}.
\begin{figure}[!h]
\begin{center}
\psfrag{f11}{\huge $3$}
\psfrag{f12}{\huge $1$}
\psfrag{f21}{\huge $5$}
\psfrag{f22}{\huge $f^2_2$}
\psfrag{f1km1}{\huge $2k-1$}
\psfrag{f2km1}{\huge $f^{k-1}_2$}
\psfrag{fsx1}{\huge $2$}
\psfrag{fsx2}{\huge $4$}
\psfrag{fsxkm1}{\huge $2k-2$}
\psfrag{c}{\huge $\cdots$}
\resizebox{10cm}{!}{\includegraphics{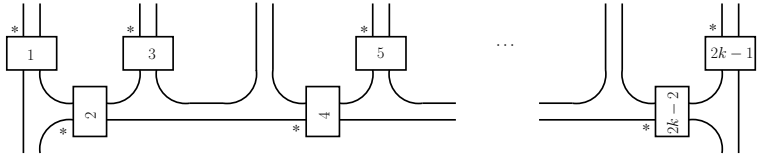}}
\end{center}
\caption{The tangle $S^{2k,+}$}
\label{fig:Stangle}
\end{figure}
Now observe that $S^{2k,+}$ is in ${\mathcal T}(2k,+)$ and thus the injectivity statement desired reduces to the Hopf algebra statement: for
every $k \geq 2$ the map, say $\eta_k : (H^* \otimes H)^{\otimes (k-1)} \rightarrow (H^*)^{\otimes (2k-1)}$ defined by
\begin{eqnarray*}
f^1 \otimes x^1 \otimes f^2 \otimes x^2 \otimes
\cdots \otimes f^{(k-1)} \otimes x^{(k-1)} \stackrel{\eta_k}{\longrightarrow}\\
f^1_2 \otimes FSx^1Sf^2_3 \otimes f^1_1Sf^2_2 \otimes f^2_4FSx^2Sf^3_3 \otimes f^2_1Sf^3_2 \otimes \cdots  \\
\cdots \otimes f^{k-2}_4FSx^{k-2}Sf^{k-1}_3 \otimes f^{k-2}_1Sf^{k-1}_2 \otimes f^{k-1}_4FSx^{k-1} \otimes f^{k-1}_1&
\end{eqnarray*}
is injective. It is this statement that we will prove by induction on $k$.

The basis case when $k=2$ asserts that $\eta_2$ defined by $\eta_2(f^1 \otimes x^1) = f^1_2 \otimes FSx^1 \otimes f^1_1$ is injective which is clear. For the inductive step, assume that $\eta_k$ is injective and observe that with the linear map $\theta: H^* \otimes H \otimes H^* \otimes H^* \rightarrow (H^*)^{\otimes 4}$ defined by $\theta(f \otimes x \otimes k \otimes q) = f_2 \otimes FSx Sk_2 \otimes f_1Sk_1 \otimes k_3q$, a short calculation shows that
$\eta_{k+1} = (\theta \otimes id^{\otimes{(2k-3)}}) \circ (id \otimes id \otimes \eta_k)$. Thus to show $\eta_{k+1}$ is injective, it suffices to see that
$\theta$ is injective.

Finally, a lengthy but complete routine calculation shows that the map
$f \otimes g \otimes k \otimes q \mapsto f_4 \otimes F(gSk_1f_3) \otimes 
Sk_3f_1 \otimes Sf_2k_2q$ is (a right inverse of, and hence) the inverse of $\theta$, completing the proof of the inductive step and hence, of the proposition.
\end{proof}

\section{Characterisation of the image}

This section will be devoted to an explicit characterisation of the image of $P(\tilde{D}(H))$ in $^{(2)}\!P(H^*)$.

Fix $k \geq 2$. Consider the (algebra) maps $\alpha,\beta : H \rightarrow End(P(H^*)_{2k,+})$ defined for $x \in H$ and $X \in P(H^*)_{2k,+}$ by Figure \ref{fig:action}.
\begin{figure}[!h]
\begin{center}
\psfrag{fx1}{\huge $Fx_1$}
\psfrag{fx2}{\huge $Fx_2$}
\psfrag{fx3}{\huge $Fx_3$}
\psfrag{fxkm1}{\huge $Fx_{k-1}$}
\psfrag{fxkm2}{\huge $Fx_{k-2}$}
\psfrag{fxk}{\huge $Fx_k$}
\psfrag{cdots}{\huge $\cdots$}
\psfrag{X}{\huge $X$}
\resizebox{12.5cm}{!}{\includegraphics{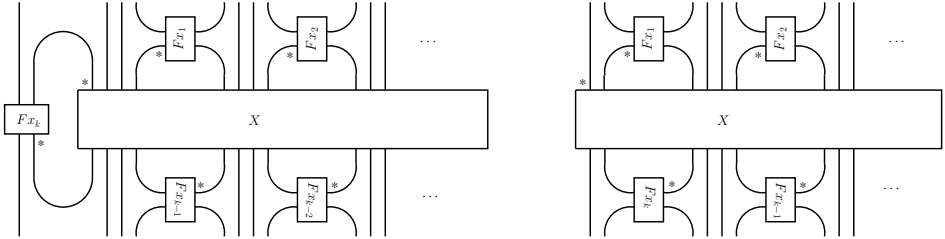}}
\end{center}
\caption{Definition of $\alpha_x(X)$ and $\beta_x(X)$}
\label{fig:action}
\end{figure}

The main result of this section is the following proposition. We will use the notation $Q$ to denote the planar subalgebra of $^{(2)}\!P(H^*)$ that is the
image of $P=P(\tilde{D}(H))$. Thus $Q_{k,\pm} \subseteq P(H^*)_{2k,+}$.

\begin{proposition}\label{charac}
 For every $k \geq 2$,
\begin{eqnarray*}
Q_{k,+} &=& \{X \in P(H^*)_{2k,+}: \alpha_h(X) = X\},\\
Q_{k,-} &=& \{X \in P(H^*)_{2k,+}: \beta_h(X) = X\}.
\end{eqnarray*}
\end{proposition}

We pave the way for a proof of this proposition by giving an alternate
description of the fixed points under $\alpha_h$. We will need some notation. 
For $x \in H$, let $\theta_k(x)$ denote the element of $P(H^*)_{4k,+}$ depicted
in Figure~\ref{fig:thetakx}. For $X \in P(H^*)_{4k,+}$ (resp. $P(H^*)_{4k+2,+}$) let $\tilde{X}\in P(H^*)_{4k+4,+}$ denote the element on the left (resp. right) in Figure \ref{fig:pushx}.
\begin{figure}[!h]
\begin{center}
\psfrag{fx1}{\huge $Fx_1$}
\psfrag{fx2}{\huge $Fx_2$}
\psfrag{fx3}{\huge $Fx_3$}
\psfrag{fxkm1}{\huge $Fx_{k-1}$}
\psfrag{fxkm2}{\huge $Fx_{k-2}$}
\psfrag{fxk}{\huge $Fx_k$}
\psfrag{cdots}{\huge $\cdots$}
\psfrag{X}{\huge $X$}
\resizebox{8cm}{!}{\includegraphics{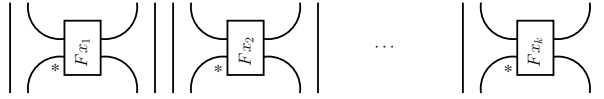}}
\end{center}
\caption{Definition of $\theta_k(x)$}
\label{fig:thetakx}
\end{figure}

\begin{figure}[!h]
\begin{center}
\psfrag{fx1}{\huge $Fx_1$}
\psfrag{fx2}{\huge $Fx_2$}
\psfrag{fx3}{\huge $Fx_3$}
\psfrag{fxkm1}{\huge $Fx_{k-1}$}
\psfrag{fxkm2}{\huge $Fx_{k-2}$}
\psfrag{fxk}{\huge $Fx_k$}
\psfrag{cdots}{\huge $\cdots$}
\psfrag{x}{\huge $X$}
\resizebox{8cm}{!}{\includegraphics{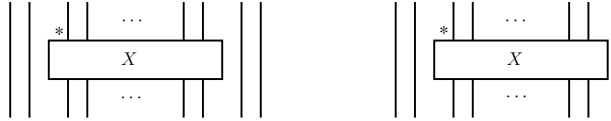}}
\end{center}
\caption{Definition of $\tilde{X}$}
\label{fig:pushx}
\end{figure}

\begin{lemma}\label{lemma:fix}
For $X \in P(H^*)_{4k,+}$ or $X \in P(H^*)_{4k+2,+}$, the following conditions are equivalent:\\
(1) $\alpha_h(X) = X$, and\\
(2) $\tilde{X}$ commutes with $\theta_{k+1}(x)$ for all $x \in H$.\\
\end{lemma}

\begin{proof}
We prove the equivalence of the conditions only for $X \in P(H^*)_{4k,+}$ leaving the case $X \in P(H^*)_{4k+2,+}$ for the reader. 
Suppose that (1) holds so that $\alpha_h(X) = X$. Then, using the definitions
of $\alpha_x(X)$ and of $\tilde{X}$, we have that $\tilde{X}$ is given by Figure \ref{fig:tildex}.
\begin{figure}[!h]
\begin{center}
\psfrag{fx1}{\huge $Fh_1$}
\psfrag{fx2}{\huge $Fh_2$}
\psfrag{fx3}{\huge $Fh_3$}
\psfrag{fxkm1}{\huge $Fh_{2k-1}$}
\psfrag{fxkm2}{\huge $Fh_{2k-2}$}
\psfrag{fxk}{\huge $Fh_{2k}$}
\psfrag{fx2k}{\huge $Fh_k$}
\psfrag{cdots}{\huge $\cdots$}
\psfrag{X}{\huge $X$}
\resizebox{8cm}{!}{\includegraphics{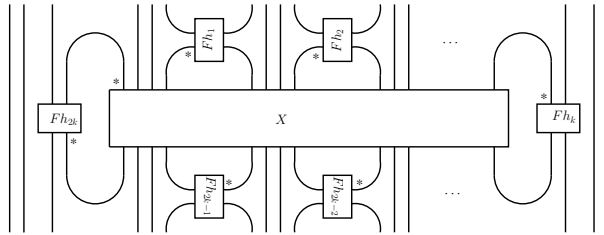}}
\end{center}
\caption{$\tilde{X}$ when $X \in P(H^*)_{4k,+}$ and $\alpha_h(X) = X$}
\label{fig:tildex}
\end{figure}
With a little manipulation and using traciality of $h$, so that $h_1 \otimes h_2 \otimes \cdots \otimes h_{2k} = h_{2k} \otimes h_1 \otimes h_2 \otimes \cdots \otimes h_{2k-1}$, we see that $\tilde{X}$ is also given by Figure \ref{fig:equivx}.
\begin{figure}[!h]
\begin{center}
\psfrag{fx1}{\huge $Fh_2$}
\psfrag{fx2}{\huge $Fh_3$}
\psfrag{fx3}{\huge $Fh_3$}
\psfrag{fxkm1}{\huge $Fh_{2k}$}
\psfrag{fxkm2}{\huge $Fh_{2k-2}$}
\psfrag{fxk}{\huge $Fh_{k+1}$}
\psfrag{fxkp1}{\huge $Fh_{k+2}$}
\psfrag{fx2k}{\huge $Fh_{2k+1}$}
\psfrag{fx2kp1}{\huge $Fh_{2k+2}$}
\psfrag{fx2kp2}{\huge $Fh_{1}$}
\psfrag{cdots}{\huge $\cdots$}
\psfrag{X}{\huge $X$}
\resizebox{8cm}{!}{\includegraphics{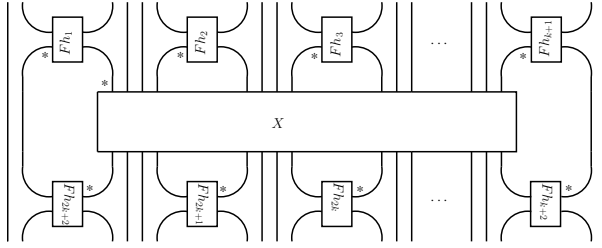}}
\end{center}
\caption{Equivalent form of $\tilde{X}$ when $X \in P(H^*)_{4k,+}$ and $\alpha_h(X) = X$}
\label{fig:equivx}
\end{figure}
Thus, $\tilde{X} =  \theta_{k+1}(h_1)\tilde{X}\theta_{k+1}(Sh_2)$.
Hence, for any $x \in H$,  $\tilde{X}\theta_{k+1}(x) = \theta_{k+1}(h_1)\tilde{X}\theta_{k+1}(Sh_2x)
=  \theta_{k+1}(xh_1)\tilde{X}\theta_{k+1}(Sh_2) = \theta_{k+1}(x)\tilde{X}$,
so that (2) is verified to hold. Conversely suppose that (2) holds for $X \in P_{4k,+}$. It follows that the element in Figure \ref{fig:equivx} equals $\tilde{X}$ and hence also the element of Figure \ref{fig:tildex}. This then implies that $\alpha_h(X) = X$, proving (1) as needed.
\end{proof}

Next, we need some preliminary commutativity statements.
\begin{lemma}\label{lemma:comm}
The following two commutativity statements hold for all $x \in H$ and all $X \in Q_{2,+}$.
\begin{figure}[!h]
\begin{center}
\psfrag{fx1}{\huge $Fx_1$}
\psfrag{fx2}{\huge $Fx_2$}
\psfrag{Fx}{\huge $Fx$}
\psfrag{fxkm1}{\huge $Fh_{2k}$}
\psfrag{fxkm2}{\huge $Fh_{2k-2}$}
\psfrag{fxk}{\huge $Fh_{k+1}$}
\psfrag{fxkp1}{\huge $Fh_{k+2}$}
\psfrag{fx2k}{\huge $Fh_{2k+1}$}
\psfrag{fx2kp1}{\huge $Fh_{2k+2}$}
\psfrag{fx2kp2}{\huge $Fh_{1}$}
\psfrag{cdots}{\huge $\cdots$}
\psfrag{X}{\huge $X$}
\psfrag{L}{\huge $\leftrightarrow$}
\resizebox{12cm}{!}{\includegraphics{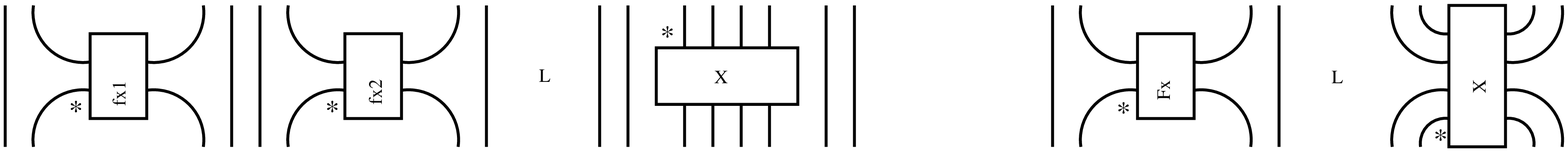}}
\end{center}
\label{fig:comm}
\end{figure}
%
%
\end{lemma}

\begin{proof}
To prove the first commutativity relation, from the form of a general generator of $Q_{2,+}$ in Figure \ref{fig:mapping0}, it suffices to see
that the commutativity in Figure \ref{fig:comm2} holds
\begin{figure}[!h]
\begin{center}
\psfrag{fx1}{\huge $Fx_1$}
\psfrag{fx2}{\huge $Fx_2$}
\psfrag{Fx}{\huge $Fx$}
\psfrag{fxkm1}{\huge $Fh_{2k}$}
\psfrag{fxkm2}{\huge $Fh_{2k-2}$}
\psfrag{fxk}{\huge $Fh_{k+1}$}
\psfrag{fxkp1}{\huge $Fh_{k+2}$}
\psfrag{fx2k}{\huge $Fh_{2k+1}$}
\psfrag{f2}{\huge $f_2$}
\psfrag{f1}{\huge $f_1$}
\psfrag{cdots}{\huge $\cdots$}
\psfrag{X}{\huge $X$}
\psfrag{L}{\huge $\leftrightarrow$}
\resizebox{7cm}{!}{\includegraphics{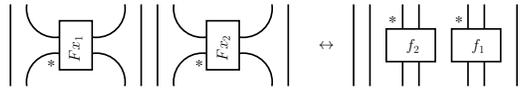}}
\end{center}
\caption{Equivalent form of the first commutativity relation}
\label{fig:comm2}
\end{figure}
To see this, note that calculation  shows that the elements $\epsilon \rtimes 1 \rtimes f_2 \rtimes 1 \rtimes f_1 \rtimes 1$  and $\epsilon \rtimes x_1 \rtimes \epsilon \rtimes 1 \rtimes  \epsilon \rtimes x_2$ of $H^* \rtimes H \rtimes H^* \rtimes H \rtimes H^* \rtimes H$ commute for all $f \in H^*$ and $x \in H$. Now applying the isomorphisms of Lemma \ref{lemma:isom} (to $P(H^*)$)
proves the desired commutativity.

As for the second commutativity relation, again from the form of a general generator of $Q_{2,+}$, it is easily seen to be equivalent to the equation in Figure \ref{fig:comm3} holding for all $f \in H^*$ and $x,a \in H$.
\begin{figure}[!h]
\begin{center}
\psfrag{fsa}{\huge $FSa$}
\psfrag{fx}{\huge $Fx$}
\psfrag{Fx}{\huge $Fx$}
\psfrag{fxkm1}{\huge $Fh_{2k}$}
\psfrag{fxkm2}{\huge $Fh_{2k-2}$}
\psfrag{fxk}{\huge $Fh_{k+1}$}
\psfrag{fxkp1}{\huge $Fh_{k+2}$}
\psfrag{fx2k}{\huge $Fh_{2k+1}$}
\psfrag{f2}{\huge $f_2$}
\psfrag{f1}{\huge $f_1$}
\psfrag{cdots}{\huge $\cdots$}
\psfrag{X}{\huge $X$}
\psfrag{L}{\huge $\leftrightarrow$}
\resizebox{7cm}{!}{\includegraphics{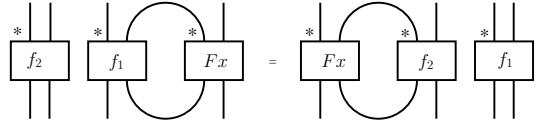}}
\end{center}
\caption{Equivalent form of the second commutativity relation}
\label{fig:comm3}
\end{figure}

Setting $Fx = g$, this is equivalent to verifying the Hopf algebraic identity
$f_2 \otimes (f_1Sg_1)(h) g_2 = (f_2Sg_2)(h) g_1 \otimes f_1$. Evaluate both sides on $a \otimes b$ to get the equivalent identity $h_1a \otimes Sh_2b = bh_1 \otimes aSh_2$ - which is easy to see.
\end{proof} 

\begin{proof}[Proof of Proposition \ref{charac}] We first prove the characterisation of $Q_{k,+}$. Since $Q$ is the image of $P(\tilde{D}(H))$,
it follows from Lemma \ref{lemma:isom} that any element  $X \in Q_{2k,+} \subseteq P(H^*)_{4k,+}$ is of the form shown in Figure \ref{fig:form},
\begin{figure}[!h]
\begin{center}
\psfrag{fx1}{\huge $Fx_1$}
\psfrag{fx2}{\huge $Fx_2$}
\psfrag{fx3}{\huge $Fh_3$}
\psfrag{fxkm1}{\huge $Fh_{2k}$}
\psfrag{fxkm2}{\huge $Fh_{2k-2}$}
\psfrag{fxk}{\huge $Fh_{k+1}$}
\psfrag{fxkp1}{\huge $Fh_{k+2}$}
\psfrag{fx2k}{\huge $Fh_{2k+1}$}
\psfrag{fx2kp1}{\huge $Fh_{2k+2}$}
\psfrag{fx2kp2}{\huge $Fh_{1}$}
\psfrag{cdots}{\huge $\cdots$}
\psfrag{X}{\huge $X_1$}
\psfrag{X=}{\huge $X = $}
\psfrag{X2}{\huge $X_2$}
\psfrag{X3}{\huge $X_3$}
\psfrag{DD}{\huge $\ddots$}
\resizebox{4cm}{!}{\includegraphics{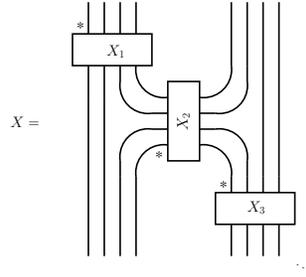}}
\end{center}
\caption{Form of $X \in Q_{2k,+}$}
\label{fig:form}
\end{figure}
\noindent
where there are $2k-1$ 4-boxes and $X_1,X_2,\cdots,X_{2k-1} \in Q_{2,+}$.
It now follows easily from Lemma \ref{lemma:comm} that $\tilde{X}$ commutes with $\theta_{k+1}(x)$ for all $x \in H$. Similarly, if $X \in Q_{2k+1,+} \subseteq P(H^*)_{4k+2,+}$, then too, $\tilde{X}$ commutes with $\theta_{k+1}(x)$ for all $x \in H$. An appeal to Lemma \ref{lemma:fix}
now shows that $Q_{k,+} \subseteq \{X \in P(H^*)_{2k,+}: \alpha_h(X) =X \}$.

To prove the reverse inclusion, it suffices by Proposition \ref{injective} to see that $dim(\{X \in P(H^*)_{2k,+}: \alpha_h(X) =X \}) \leq n^{2k-2}$.
Consider the tangle $V^{2k,+}$ of Figure \ref{fig:tanglev}.
\begin{figure}[!h]
\begin{center}
\psfrag{1}{\large $1$}
\psfrag{2}{\large $2$}
\psfrag{3}{\large $3$}
\psfrag{k-2}{\large $k-2$}
\psfrag{k-1}{\large $k-1$}
\psfrag{k}{\large $k$}
\psfrag{2k-4}{\large $2k-4$}
\psfrag{2k-3}{\large $2k-3$}
\psfrag{2k-2}{\large $2k-2$}
\psfrag{2k-1}{\Large $2k-1$}
\psfrag{c}{\huge $\cdots$}
\resizebox{9cm}{!}{\includegraphics{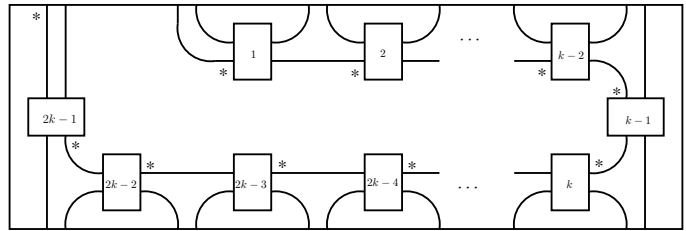}}
\end{center}
\caption{The tangle $V^{2k,+}$}
\label{fig:tanglev}
\end{figure}
Note that $V^{2k,+} \in {\mathcal T}(2k,+)$ and hence induces a linear isomorphism from $(H^*)^{\otimes (2k-1)} \rightarrow P(H^*)_{2k,+}$.
Further, we see that 
\begin{eqnarray*}\alpha_h(Z_{V^{2k,+}}(Fa^1 \otimes Fa^2 \otimes \cdots \otimes Fa^{2k-1})) = \\
Z_{V^{2k,+}}(F(h_1a^1) \otimes Fa^2 \otimes F(h_2a^3) \otimes Fa^4 \otimes \cdots \otimes Fa^{2k-2} \otimes F(h_ka^{2k-1})).\end{eqnarray*}

Thus it suffices to see that $dim(\{ a^1 \otimes \cdots \otimes a^{2k-1} \in H^{\otimes (2k-1)} : a^1 \otimes \cdots \otimes a^{2k-1} = h_1a^1 \otimes a^2 \otimes h_2a^3 \otimes a^4 \otimes \cdots \otimes a^{2k-2} \otimes  h_ka^{2k-1} \} \leq n^{2k-2}$ or equivalently that 
$dim(\{ x^1 \otimes \cdots \otimes x^{k} \in H^{\otimes k} : x^1 \otimes \cdots \otimes x^{k} = h_1x^1 \otimes  h_2x^2 \otimes  \cdots \otimes  h_kx^{k} \} \leq n^{k-1}$. Now observe that if $x^1 \otimes \cdots \otimes x^{k} = h_1x^1 \otimes  h_2x^2 \otimes  \cdots \otimes  h_kx^{k}$, then
\begin{eqnarray*}
x^1 \otimes \cdots \otimes x^{k} &=& h_1x^1 \otimes  h_2x^2 \otimes  \cdots \otimes  h_kx^{k}\\
&=& h_1x^1 \otimes  \Delta_{k-1}(h_2)(x^2 \otimes \cdots \otimes x^k)\\
&=&h_1 \otimes  \Delta_{k-1}(h_2Sx^1)(x^2 \otimes \cdots \otimes x^k)\\
&=& h_1 \otimes  h_2Sx^1_{k-1}x^2 \otimes  \cdots \otimes  h_kSx^1_1x^{k}
\end{eqnarray*}
This is clearly in the image of the map $H^{\otimes k-1} \rightarrow H^{\otimes k}$ given by $z^1 \otimes \cdots \otimes z^{k-1} 
\mapsto h_1 \otimes h_2z^2 \otimes \cdots \otimes h_kz^{k-1}$ and so the required dimension estimate follows, establishing the characterisation of $Q_{k,+}$.

Now note that, $X \in Q_{k,-} \Leftrightarrow Z_R(X) \in Q_{k,+}$ (where $R$ is the one-rotation tangle on $(k,-)$ boxes, since $Q$ is a planar subalgebra  of $^{(2)}\!P(H^*)$). The action of $R$ on $Q_{k,-}$ is given by
the restriction of the action of the two-rotation tangle on $P(H^*)_{2k,+}$. Hence the 
asserted characterisation of $Q_{k,-}$ follows directly from that of $Q_{k,+}$.
\end{proof}

\section{The main theorem}
We collect the results of the previous statements into a single main theorem.

\begin{theorem}\label{thm:main}
Let $H$ be a finite-dimensional, semisimple and cosemisimple Hopf algebra over $\k$ of dimension $n = \delta^2$
with Drinfeld double $\tilde{D}(H)$. The map
$$P(\tilde{D}(H))_{2,+}\  \longrightarrow\  {}^{(2)}\!P(H^*)_{2,+}$$
defined in Proposition \ref{main}  extends to an  injective planar algebra morphism $P(\tilde{D}(H))\  \longrightarrow\  {}^{(2)}\!P(H^*)$
whose image $Q$ is characterised as follows: $Q_{k,+}$ (resp. $Q_{k,-}$) is the set of all $X \in P_{2k,+}$  such the element on the left (resp. right)
in Figure \ref{fig:charac2} equals $X$.
\end{theorem}

\begin{figure}[!h]
\begin{center}
\psfrag{1}{\huge $Fh_1$}
\psfrag{2}{\huge $Fh_2$}
\psfrag{3}{\huge $Fh_3$}
\psfrag{km1}{\huge $Fh_{k-1}$}
\psfrag{k}{\huge $Fh_k$}
\psfrag{2k-4}{\large $2k-4$}
\psfrag{2k-3}{\large $2k-3$}
\psfrag{2k-2}{\large $2k-2$}
\psfrag{2k-1}{\Large $2k-1$}
\psfrag{X}{\huge $X$}
\resizebox{12cm}{!}{\includegraphics{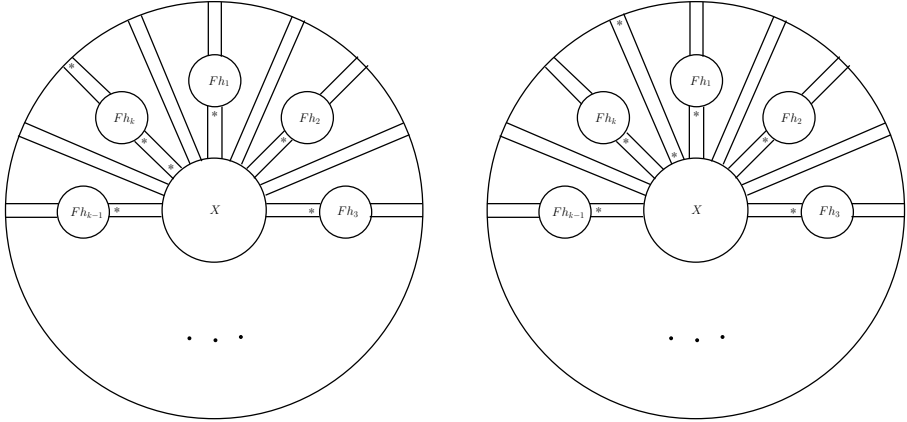}}
\end{center}
\caption{Characterisation of the image}
\label{fig:charac2}
\end{figure}

\begin{proof}
This follows from Propositions \ref{main}, \ref{injective} and \ref{charac},
after observing that Figures \ref{fig:action} and \ref{fig:charac2} are equivalent.
\end{proof}

We conclude with a few remarks. First, the main theorem above also allows us to conclude that there is an explicitly characterised  planar subalgebra of ${}^{(2)}\!P(H^{op})$ that is isomorphic to $P(\tilde{D}(H))$. This is because $\tilde{D}(H)$ and $\tilde{D}((H^{op})^{*})$ are isomorphic. This leads to the natural question as to whether in some precise sense, $P(\tilde{D}(H))$ is the `smallest' planar algebra that maps into both ${}^{(2)}\!P(H^*)$ and ${}^{(2)}\!P(H^{op})$. Finally, we hope to explore the subfactor theoretic aspects
of this result and some generalisations in a future publication.

\end{document}